\documentclass[sn-mathphys,Numbered]{sn-jnl}
\usepackage{graphicx}
\usepackage{multirow}
\usepackage{amsmath,amssymb,amsfonts}
\usepackage{amsthm}
\usepackage{mathrsfs}
\usepackage[title]{appendix}
\usepackage{xcolor}
\usepackage{textcomp}
\usepackage{manyfoot}
\usepackage{booktabs}
\usepackage{algorithm}
\usepackage{algorithmicx}
\usepackage{algpseudocode}
\usepackage{listings}
\usepackage{hyperref}
\usepackage{natbib}
\theoremstyle{thmstyleone}
\newtheorem{theorem}{Theorem}
\newtheorem{proposition}[theorem]{Proposition}
\theoremstyle{thmstyletwo}
\newtheorem{example}{Example}
\newtheorem{remark}{Remark}
\newtheorem{corollary}[theorem]{Corollary}
\theoremstyle{thmstylethree}
\newtheorem{definition}{Definition}
\newtheorem{conjecture}[theorem]{Conjecture}
\newtheorem{lemma}[theorem]{Lemma}
\begin{document}

\title[A Note on Quaternions over Commutative Rings ]{A Note on Quaternions over Commutative Rings }

\author*[1]{\fnm{A.} \sur{Telveenus}}\email{telveenusa@gmail.com}

\affil*[1]{\orgdiv{Former Lecturer, International Study Centre}, \orgname{Kingston University}, \orgaddress{\street{Kingston Hill Campus}, \city{London}, \postcode{KT2 7LB},  \country{UK}}  \orgdiv{and Former HoD Mathematics, Fatima Mata National College}, \orgname{University of Kerala}, \orgaddress{\street{Kollam}, \city{Kerala},  \country{India}} }

\abstract{The ring of quaternions is defined over the field of real numbers. This article extends that framework by defining quaternions over a class of commutative unital rings  and generalising several of their properties. It also studies the properties of the unit group of quaternions over commutative rings with special attention to halidon rings. }



\maketitle
\section{Introduction}
In 1843, the Irish mathematician and physicist William Rowan Hamilton introduced the concept of quaternions as an extension of complex numbers, where multiplication is commutative, whereas the multiplication of quaternions is non-commutative.   In his language, the sum of a scalar and a vector is a quaternion(\cite{h}, Book I, Chapter I,page 11).
Let $\mathbb{R}$ be the field of real numbers and $\mathbb{H}=\{ q=w+xi+yj+zk \ | w,x,y,z\in \mathbb{R}\}$. The addition is defined as component wise and the multiplication in $\mathbb{H}$ is defined by the following rules:\begin{enumerate}
                                 \item  $i^{2}=j^{2}=k^{2}=-1;$
                                 \item $ ij=k, \ jk=i, \ ki=j \ ;$
                                 \item $ ji=-k, \ kj=-i, \ ik=-j.  $ (\cite{h}, Book II, Chapter I,page 158)
                               \end{enumerate}
Clearly $\mathbb{H}$ is division ring over $\mathbb{R}$. $w$ and $xi+yj+zk$ are called the \textit{real} and \textit{imaginary} parts of $q$ or \textit{scalar quaternion} and \textit{vector quaternion} respectively. The rule 2 can be interpreted as $i \rightarrow j \rightarrow k \rightarrow i$ in a positive orientation(anticlockwise orientation). When the orientation is opposite, we get negative results as in rule 3. According to the above rules, the multiplication (Hamilton) of two elements $q_{1}=w_{1}+x_{1}i+y_{1}j+z_{1}k$ and $q_{2}=w_{2}+x_{2}i+y_{2}j+z_{2}k$ is given by 
$q_{1}q_{2}=(w_{1}w_{2}-x_{}x_{2}-y_{1}y_{2}-z_{1}z_{2})+(w_{1}x_{2}+w_{2}x_{1}+y_{1}z_{2}-z_{1}y_{2})i+(w_{1}y_{2}+w_{2}y_{1}-x_{1}z_{2}+z_{1}x_{2})j+(w_{1}z_{2}+z_{1}w_{2}+x_{1}y_{2}-x_{2}y_{1})k$
$q^{*}=w-xi-yj-zk$ is defined as the conjugate of $q=w+xi+yj+zk$ and the norm of $q$,$\mid\mid q \mid\mid$, is defined by the equation $\mid\mid q \mid\mid^{2}=qq^{*}=w^{2}+x^{2}+y^{2}+z^{2}$. If 
$\mid\mid q \mid\mid \neq 0$, then $q^{-1}=\dfrac{1}{w^{2}+x^{2}+y^{2}+z^{2}}(w-xi-yj-zk)$. If $q=w+xi+yj+zk=w+u$ with $\mid\mid u \mid\mid=1$, then $u^{2}=-1$ and thereby $u^{4}=1$. 
\section{Quaternions over Commutative Rings}
Let $\mathbf{R}$ be a commutative unital ring and let $U(\mathbf{R})$ be the unit group of $\mathbf{R}$. Let $NU(\mathbf{R})$ be the non-unit elements of $\mathbf{R}$. 
We define the ring of quaternions over $\mathbf{R}$ as $\mathcal{H_{\mathbf{R}}}=\{ q=w+xi+yj+zk \ | w,x,y,z\in \mathbf{R}\}$ and the rules of addition and multiplication remain the same as in the standard ring of quaternions. There are no changes in the definitions of conjugate and norm. 
\begin{proposition}
An element $q$ in $\mathcal{H_{\mathbf{R}}}$ is invertible in $\mathcal{H_{\mathbf{R}}}$ if and only if $\mid\mid q \mid\mid^{2} \in U(\mathbf{R})$. 
\end{proposition}
\begin{proof}
$\mid\mid q \mid\mid^{2}=qq^{*}$ implies $q^{-1}=\mid\mid q \mid\mid^{-2}q^{*}$ and its existence is possible if and only if $\mid\mid q \mid\mid^{2} \in U(\mathbf{R})$ and hence the proof. 
\end{proof}
\begin{theorem}
Let $\mathcal{H_{\mathbf{R}}}=\{ q=w+xi+yj+zk \ | w,x,y,z\in \mathbf{R}\}$ be the ring of quaternions over $\mathbf{R}$. Then there exists an induced ring of quaternions $\mathcal{H_{\mathbf{R}}^{\perp}}$ with orientation just opposite to $\mathcal{H_{\mathbf{R}}}$ such that $\mathcal{H}_{\mathbf{R}}\cong\mathcal{H_{\mathbf{R}}^{\perp}}$. 
\end{theorem}
\begin{proof}
We define $\mathcal{H_{\mathbf{R}}^{\perp}}=\{ Q=wI+x\Gamma_{1}+y\Gamma_{2}+z\Gamma_{3} \ | w,x,y,z\in \mathbf{R}\}$, where \\
$I=\left(
\begin{array}{cccc}
1 & 0 & 0 & 0 \\
0 & 1 & 0 & 0 \\
0 & 0 & 1 & 0 \\
0 & 0 & 0 & 1 \\
\end{array}
\right)
$,
$\Gamma_{1}=\left(
\begin{array}{cccc}
0 & -1 & 0 & 0 \\
1 & 0 & 0 & 0 \\
0 & 0 & 0 & 1 \\
0 & 0 & -1 & 0 \\
\end{array}
\right)
$,
$\Gamma_{2}=\left(
\begin{array}{cccc}
0 & 0 & -1 & 0 \\
0 & 0 & 0 & -1 \\
1 & 0 & 0 & 0 \\
0 & 1 & 0 & 0 \\
\end{array}
\right)
$,
$\Gamma_{3}=\left(
\begin{array}{cccc}
0 & 0 & 0 & -1 \\
0 & 0 & 1 & 0 \\
0 & -1 & 0 & 0 \\
1 & 0 & 0 & 0 \\
\end{array}
\right)
$.
A bit lengthy calculation shows that $\mathcal{H_{\mathbf{R}}^{\perp}}$ satisfies all rules of addition and multiplication except the interchange of rule 2 and rule 3 of multiplication. 
Let $f:\mathcal{H_{\mathbf{R}}} \mapsto \mathcal{H_{\mathbf{R}}^{\perp}}$ be defined by $f(w+xi+yj+zk)=wI-x\Gamma_{1}-y\Gamma_{2}-z\Gamma_{3}$. This map is clearly a bijective map preserving the operation  addition and $f(1)=I$. Let us show that f preserves the multiplication as well. \\
Let $q_{1}=w_{1}+x_{1}i+y_{1}j+z_{1}k$, $q_{2}=w_{2}+x_{2}i+y_{2}j+z_{2}k$ be any two elements in $\mathcal{H}$. Then $f(q_{1})=w_{1}I-x_{1}\Gamma_{1}-y_{1}\Gamma_{2}-z_{1}\Gamma_{3}$ and $f(q_{2})=w_{2}I-x_{2}\Gamma_{1}-y_{2}\Gamma_{2}-z_{2}\Gamma_{3}$. \\
$q_{1}q_{2}=(w_{1}w_{2}-x_{}x_{2}-y_{1}y_{2}-z_{1}z_{2})+(w_{1}x_{2}+w_{2}x_{1}+y_{1}z_{2}-z_{1}y_{2})i+(w_{1}y_{2}+w_{2}y_{1}-x_{1}z_{2}+z_{1}x_{2})j
+(w_{1}z_{2}+z_{1}w_{2}+x_{1}y_{2}-x_{2}y_{1})k.$ \\
$f(q_{1}q_{2})=(w_{1}w_{2}-x_{}x_{2}-y_{1}y_{2}-z_{1}z_{2})I-(w_{1}x_{2}+w_{2}x_{1}+y_{1}z_{2}-z_{1}y_{2})\Gamma_{1}-(w_{1}y_{2}+w_{2}y_{1}-x_{1}z_{2}+z_{1}x_{2})\Gamma_{2}
+(w_{1}z_{2}+z_{1}w_{2}+x_{1}y_{2}-x_{2}y_{1})\Gamma_{3}.$ \\
\noindent
Using the relations $\Gamma_{1}^{2}=\Gamma_{2}^{2}=\Gamma_{3}^{2}=-1,\Gamma_{1}\Gamma_{2}=-\Gamma_{3},\Gamma_{2}\Gamma_{3}=-\Gamma_{1},\Gamma_{3}\Gamma_{1}=-\Gamma_{2},
\Gamma_{2}\Gamma_{1}=\Gamma_{3},\Gamma_{3}\Gamma_{1}=\Gamma_{1},\Gamma_{1}\Gamma_{3}=\Gamma_{2}$, we have 
$f(q_{1})f(q_{2})=(w_{1}I-x_{1}\Gamma_{1}-y_{1}\Gamma_{2}-z_{1}\Gamma_{3})(w_{2}I-x_{2}\Gamma_{1}-y_{2}\Gamma_{2}-z_{2}\Gamma_{3})=(w_{1}w_{2}-x_{1}x_{2}-y_{1}y_{2}-z_{1}z_{2})I
-(w_{1}x_{2}+w_{2}x_{1}+y_{1}z_{2}-z_{1}y_{2})\Gamma_{1}-(w_{1}y_{2}+w_{2}y_{1}-x_{1}z_{2}+z_{1}x_{2})\Gamma_{2}
+(w_{1}z_{2}+z_{1}w_{2}+x_{1}y_{2}-x_{2}y_{1})\Gamma_{3}=f(q_{1}q_{2}).$  \\
$ \therefore \quad \quad \mathcal{H_{\mathbf{R}}}\cong\mathcal{H_{\mathbf{R}}^{\perp}} $
\end{proof}
\begin{theorem}
If $q=w+xi+yj+zk=w+u \in \mathcal{H_{\mathbf{R}}}$ with $\mid\mid u \mid\mid=1$, then $q^{n}=q_{n}+q_{n}^{\prime}u$ where $\left( \begin{array}{c} q_{n} \\
q_{n}^{\prime}\end{array}\right)$=$\left(\begin{array}{cc} a & -1 \\ 1 & a\end{array} \right)^{n-1}$ $\left( \begin{array}{c} a \\
1 \end{array}\right)$ and $n$ is a positive integer.
\end{theorem}
\begin{proof}
Let us apply the method of mathematical induction to prove this theorem. \\
When $n$=1, $q^{1}=q_{1}+q_{1}^{\prime}u=a+u$ which is true. \\
Suppose the statement is true for n=s. \\ Then $q^{s}=q_{s}+q_{s}^{\prime}u$ where $\left( \begin{array}{c} q_{s} \\
q_{s}^{\prime}\end{array}\right)$=$\left(\begin{array}{cc} a & -1 \\ 1 & a\end{array} \right)^{s-1}$ $\left( \begin{array}{c} a \\
1 \end{array}\right).$\\
$$q^{s+1}=(q_{s}+q_{s}^{\prime})(a+u) \quad \ \ =aq_{s}-q_{s}^{\prime}+(q_{s}+aq_{s}^{\prime})u  \quad \text{using}  \ u^{2}=-1$$
$$=q_{s+1}+q_{s1}^{\prime}$$ \\
where $\left( \begin{array}{c} q_{s+1} \\
q_{s+1}^{\prime}\end{array}\right)$=$\left(\begin{array}{cc} a & -1 \\ 1 & a\end{array} \right)$ $\left( \begin{array}{c} q_{s} \\
q_{s}^{\prime} \end{array}\right)$ = $\left( \begin{array}{c} q_{s} \\
q_{s}^{\prime}\end{array}\right)$=$\left(\begin{array}{cc} a & -1 \\ 1 & a\end{array} \right)^{s}$ $\left( \begin{array}{c} a \\
1 \end{array}\right).$\\
$\therefore \ $  It is true for $n=s+1$ \\
By mathematical induction the statement is true for all positive integer $n$ and hence the theorem. 
\end{proof}Next theorem deals with the computation of $U(\mathcal{H_{\mathbf{R}}})$.
\begin{theorem} \label{t1}
Let $\mathbf{R}$ be a finite commutative unital ring. Then $U(\mathcal{H_{\mathbf{R}}})=\{ q=w+xi+yj+zk \in \mathcal{H_{\mathbf{R}}} \ | q^{s}=1, \ \text{for some positive integer} \ s \}$ and $q^{s}$ can be computed using the following recurrence formula:\\
$$q^{s}=q_{s}+q_{s}^{\prime}i+q_{s}^{\prime \prime}j+q_{s}^{\prime \prime \prime }k $$ where \\
 
 $$ \left( \begin{array}{c}
 q_{s} \\
 q_{s}^{\prime} \\
 q_{s}^{\prime \prime} \\
 q_{s}^{\prime \prime \prime }
 \end{array}\right)
 = \left(\begin{array}{cccc} 
w & -x & -y& -z \\
x & w & z & -y \\
y & -z & w & x \\
z & y & -x & w 
\end{array} \right)
 \left( \begin{array}{c}
 q_{s-1} \\
 q_{s-1}^{\prime} \\
 q_{s-1}^{\prime \prime} \\
 q_{s-1}^{\prime \prime \prime }

  \end{array}\right)$$

 with  
$\left( \begin{array}{c}
 q_{1} \\
 q_{1}^{\prime} \\
 q_{1}^{\prime \prime} \\
 q_{1}^{\prime \prime \prime }

  \end{array}\right)$= $\left( \begin{array}{c}
 w \\
 x\\
 y \\
 z

  \end{array}\right).$
  
\end{theorem}
\begin{proof}
$q^{s}=1 \implies q^{-1}=q^{s-1} \implies q \in U(\mathcal{H_{\mathbf{R}}}).$ \\
$q^{s}=q^{s-1}q=(q_{s-1}+q_{s-1}^{\prime}i+q_{s-1}^{\prime \prime}j+q_{s-1}^{\prime \prime \prime }k)(w+xi+yj+zk) \\
= (wq_{s-1}-xq_{s-1}^{\prime}-yq_{s-1}^{\prime \prime}-zq_{s-1}^{\prime \prime \prime }))+(xq_{s-1}wq_{s-1}^{\prime}+zq_{s-1}^{\prime \prime}-yq_{s-1}^{\prime \prime \prime }))i
+(yq_{s-1}-zq_{s-1}^{\prime}+wq_{s-1}^{\prime \prime}j+xq_{s-1}^{\prime \prime \prime }))j+(zq_{s-1}+yq_{s-1}^{\prime}-xq_{s-1}^{\prime \prime}+wq_{s-1}^{\prime \prime \prime }))k \\ 
=q_{s}+q_{s}^{\prime}i+q_{s}^{\prime \prime}j+q_{s}^{\prime \prime \prime }k $  \\ 
$\therefore  $$ \left( \begin{array}{c}
 q_{s} \\
 q_{s}^{\prime} \\
 q_{s}^{\prime \prime} \\
 q_{s}^{\prime \prime \prime }
 \end{array}\right)
 = \left(\begin{array}{cccc} 
w & -x & -y& -z \\
x & w & z & -y \\
y & -z & w & x \\
z & y & -x & w 
\end{array} \right)
 \left( \begin{array}{c}
 q_{s-1} \\
 q_{s-1}^{\prime} \\
 q_{s-1}^{\prime \prime} \\
 q_{s-1}^{\prime \prime \prime }

  \end{array}\right)$
\end{proof}
\begin{proposition}
Let $\mathbf{R}=\mathbb{Z}_{n}$, the ring of integers modulo $n$. If $n=2^{t}$, then $|U(\mathcal{H_{\mathbf{R}}})|=\phi(|\mathcal{H_{\mathbf{R}}}|)$.
\end{proposition}
\begin{proof}
In $\mathbb{Z}_{2^{t}}$, $|U(\mathbb{Z}_{2^{t}})|=|NU(\mathbb{Z}_{2^{t}})| \implies |U(\mathcal{H_{\mathbf{R}}})|=|NU(\mathcal{H_{\mathbf{R}}})|$. But 
$|U(\mathcal{H_{\mathbf{R}}})|+|NU(\mathcal{H_{\mathbf{R}}})|=2^{4t}$. This implies $|U(\mathcal{H_{\mathbf{R}}})|=2^{4t-1}=\phi(2^{4t})=\phi(|\mathcal{H_{\mathbf{R}}}|)$
\end{proof}
An element $q$ in $\mathcal{H_{\mathbf{R}}}$ is said to be a \textit{nil potent} element with index $s$ if $s$ is the least positive integer such that $q^{s}=0.$ The set of all nil potent elements in $\mathcal{H_{\mathbf{R}}}$ is denoted by $N(\mathcal{H_{\mathbf{R}}})$. 
\begin{theorem}
For primes $p>3$, $q=w+xi+yj+zk \in N(\mathcal{H_{\mathbf{\mathbb{Z}_{p}}}})$ $\iff$ $w=0$ and $x^{2}+y^{2}+z^{2}=0$. Its index is 2 and $|N(\mathcal{H_{\mathbf{\mathbb{Z}_{p}}}})|=p^{2}.$ 
\end{theorem}
\begin{proof}
We have $q^{2}=(w^{2}-x^{2}-y^{2}-z^{2})+2wxi+2wyj+2wzk.$ Suppose that $q^{2}=0$. Then $w^{2}-x^{2}-y^{2}-z^{2}=0, \ 2wx=0, \ 2wy=0, \ 2wz=0$. If $w \neq 0, $ then $w^{-1}$ exists in $\mathbb{_{p}}$. This $\implies$ $x=y=z=0$ and the first equation gives $w=0$ which is a contradiction. $\therefore$ $w=0 $ and $x^{2}+y^{2}+z^{2}=0$. If $q^{2} \neq 0$ and assume that $q^{3}=0$, then we have the following equations: \\
$$w(w^{2}-3x^{2}-3y^{2}-3z^{2})=0$$ 
$$x(3w^{2}-x^{2}-y^{2}-z^{2})=0$$ 
$$y(3w^{2}-x^{2}-y^{2}-z^{2})=0$$ 
$$z(3w^{2}-x^{2}-y^{2}-z^{2})=0$$ 
Let us assume that $w \neq 0$. This implies $w^{2}=x^{}+y^{}+z^{2}.$ 
For nontrivial nil potent elements, at least one of the $x, \ y \ ,z$  must be non zero. For this let us take $x\neq 0$. This implies $$3w^{2}=x^{}+y^{2}+z^{2} \implies 9w^{2}=3w^{2}$$
$$\implies 6w^{2}=0 \implies w^{2}=0 \implies w=0$$ as $6$ has multiplicative inverse in $\mathbb{Z}_{p}$ and $p>3.$ It is again a contradiction. $\therefore$ $w$ must be zero. When $w=0$, automatically, $x^{2}+y^{2}+z^{2}=0$. This is the required condition for $q^{2}=0$. Even if $q^{3}=0$, before this $q^{2}=0$ must happen. $\therefore $ 2 is the least positive integer such that $q^{2}=0$. So the index is $2$.  \\
It is well known that, for an odd prime $p$, $\sqrt{-1} mod \ p $ exists only when $p=4t+1$ for some integer $t>0$ or the equation $x^{2}+1 \equiv 0 \ mod \ p$ has solutions $\iff$ $p=4t+1$ and the solutions are $\pm (\frac{p-1}{2})!$. \\
Suppose that $p=4t+1$. We have $w=0$ and $x^{2}+y^{2}+z^{2}=0$. If $x=0$, then $y=\pm\sqrt{-1}z$. Thus we have $2(p-1)$ non trivial triplets $(0,y,z)$ as $y$ can take $(p-1)$ values and $z$ can take $2$ values. For $x \neq 0$, we can assign $(p-1)(p-1)$ non trivial triplets $(x,y,z )$ as $z$ takes only one value. \\
$\therefore $ total number of nil potent elements $= (p-1)(p-1)+2(p-1)+1=p^2 $, the last number 1 is the trivial nil potent element. Thus we have $|N(\mathcal{H_{\mathbf{\mathbb{Z}_{p}}}})|=p^{2}.$ When $p\neq 4t+1$, there are $(\dfrac{p-1}{2})$ numbers in $\mathbb{Z}_{p}$ which has square roots and each square will take $2$ values. In the calculation of $z^{2}=-x^{2}-y^{2}$, each square root will appear in $(p+1)$ places (this is not true for $p=4t+1$). Hence the number of nil potent elements = $(\dfrac{p-1}{2}) \times 2 \times (p+1) +1 =p^{2}$. The last one is for the trivial nil potent element. Hence the theorem.
\end{proof}
\begin{remark}
The above theorem is true for $p=3$ as well. See programme $1$ in appendix. 
\end{remark}
This example verifies that the theorem stated above holds for all prime numbers $p > 2$ . The verification is carried out only for $p=3$ and $p=5$ using Programme 1, which is provided in the appendix. Readers can verify the results for any prime greater than $2$, but some RAM restrictions are there.
\begin{example}
\begin{verbatim}
To find the nil potent elements  q=a+bi+cj+dk.
Enter n=
3
q=0+0i + 0j +  0k is of order 1

q=0+1i + 1j +  1k is of order 2

q=0+1i + 1j +  2k is of order 2

q=0+1i + 2j +  1k is of order 2

q=0+1i + 2j +  2k is of order 2

q=0+2i + 1j +  1k is of order 2

q=0+2i + 1j +  2k is of order 2

q=0+2i + 2j +  1k is of order 2

q=0+2i + 2j +  2k is of order 2

number of nil potent elements=9=3^2
\end{verbatim}
\begin{verbatim}
To find the nil potent elements  q=a+bi+cj+dk.
Enter n=
5
q=0+0i + 0j +  0k is of order 1

q=0+0i + 1j +  2k is of order 2

q=0+0i + 1j +  3k is of order 2

q=0+0i + 2j +  1k is of order 2

q=0+0i + 2j +  4k is of order 2

q=0+0i + 3j +  1k is of order 2

q=0+0i + 3j +  4k is of order 2

q=0+0i + 4j +  2k is of order 2

q=0+0i + 4j +  3k is of order 2

q=0+1i + 0j +  2k is of order 2

q=0+1i + 0j +  3k is of order 2

q=0+1i + 2j +  0k is of order 2

q=0+1i + 3j +  0k is of order 2

q=0+2i + 0j +  1k is of order 2

q=0+2i + 0j +  4k is of order 2

q=0+2i + 1j +  0k is of order 2

q=0+2i + 4j +  0k is of order 2

q=0+3i + 0j +  1k is of order 2

q=0+3i + 0j +  4k is of order 2

q=0+3i + 1j +  0k is of order 2

q=0+3i + 4j +  0k is of order 2

q=0+4i + 0j +  2k is of order 2

q=0+4i + 0j +  3k is of order 2

q=0+4i + 2j +  0k is of order 2

q=0+4i + 3j +  0k is of order 2

number of nil potent elements=25=5^2
\end{verbatim}
\end{example}
An element $b = b^{n}$ is called \textit{n-idempotent} and more specifically, for $n = 3$, such an element
is called a \textit{tripotent}. A tripotent will be called \textit{genuine}
if it is not idempotent (so=$ 0,1$), not the negative of an idempotent nor an involution(unit of order $2$).
\begin{proposition} \cite{gc}
If  $q=w+xi+yj+zk \in \mathcal{H_{\mathbb{R}}}$  is a tripotent, then
\begin{enumerate}
  \item there is no genuine tripotent with $w=0$; \\
  \item $q^{2}$ and $1-q^{2}$ are both idempotents; the converse fails ($2$ is not tripotent in $\mathbb{Z}_{12}$); \\
  \item $q^{2}$ and $1-q^{2}$ are orthogonal; \\
  \item $1 -q -q^{2}$ is an involution; \\
  \item  and if $q$ is not an idempotent, then $1 -q -q^{2}$ is an involution; \\
  \item $\mathcal{H_{\mathbb{R}}}q^{2}\oplus\mathcal{H_{\mathbb{R}}}(1-q^{2}) = \mathcal{H_{\mathbb{R}}}$ is the Pierce decomposition into left (principal) ideals. A symmetric decomposition into right ideals also holds.
\end{enumerate}
\end{proposition}
\begin{proof}
\noindent
\begin{enumerate}
  \item When $w=0$ $q=xi+yj+zk.$ This $ \implies$ $q^{2}=-x^{2}-y^{2}-z^{2}$ \\ and $q^{3}=q$ $\implies$ $(-x^{2}-y^{2}-z^{2})q=q$  $\implies$ $(-x^{2}-y^{2}-z^{2})=1 \implies q^{2} =1$. This is an involution. $\therefore \  \ q $ cannot be genuine.
      \item The statement is evidently true.. 
  \item $q^{2}(1-q^{2})=q^{2} -q^{4}=q^{2}-q^{2}=0$ as $q^{3}=q.$ \\
  $\therefore $ $q^{2}$ and $1-q^{2}$ are orthogonal.
  \item $(1 -q -q^{2})^{2}=1+q^{2}+q^{4}-2q-2q^{2}+2q^{3}=1+q^{2}+q^{2}-2q-2q^{2}+2q=1.$ \\
  $\therefore $  $1 -q -q^{2}$ is an involution.
  \item Same as above as we did not use the property of idempotent above.
  \item Since the sum of the orthogonal idempotents is $1$, the statement follows. 
  
\end{enumerate}
\end{proof}
\begin{proposition} \label{rd21}
Let $P=w_{1}+x_{1}i+y_{1}j+z_{1}k$ and $Q=w_{2}+x_{2}i+y_{2}j+z_{2}k$ be any two elements of $\mathcal{H}_{\mathbf{R}}$, where $\mathbf{R}$ is a commutative unital ring and if $\mathbf{R }$ is an integral domain or $2 \in U(\mathbf{R})$, then $PQ=QP$ $\iff$ the following simultaneous equations are satisfied:\\ $$y_{1}z_{2}=z_{1}y_{2}$$ $$z_{1}x_{2}=x_{1}z_{2}$$ $$x_{1}y_{2}=y_{1}x_{2}$$. 
\end{proposition}
\begin{proof}
$PQ=(w_{1}w_{2}-x_{1}x_{2}-y_{1}y_{2}-z_{1}z_{2})+(w_{1}x_{2}+x_{1}w_{2}+y_{1}z_{2}-z_{1}y_{2})i+(w_{1}y_{2}+y_{1}w_{2}-x_{1}z_{2}+z_{1}x_{2})j+(w_{1}z_{2}+z_{1}w_{2}+x_{1}y_{2}-y_{1}x_{2})k$ \\
$QP=(w_{1}w_{2}-x_{1}x_{2}-y_{1}y_{2}-z_{1}z_{2})+(w_{1}x_{2}+x_{1}w_{2}-y_{1}z_{2}+z_{1}y_{2})i+(w_{1}y_{2}+y_{1}w_{2}+x_{1}z_{2}-z_{1}x_{2})j+(w_{1}z_{2}+z_{1}w_{2}-x_{1}y_{2}+y_{1}x_{2})k$ \\
$PQ-QP=2(y_{1}z_{2}-z_{1}y_{2})i+2(z_{1}x_{2}-x_{1}z_{2})j+2(x_{1}y_{2}-y_{1}x_{2})k$ \\ $PQ=QP \iff $ 
$$2(y_{1}z_{2}-z_{1}y_{2})=0$$ 
$$2(z_{1}x_{2}-x_{1}z_{2})=0$$ 
$$2(x_{1}y_{2}-y_{1}x_{2})=0$$ Since $\mathbf{R }$ is an integral domain or $2 \in U(\mathbf{R})$, the equations follow.
\end{proof}
We define $\mathcal{S}_{\mathbf{R}}=\{ w+a(i+j+k) | w,a \in \mathbf{R}\}$. Clearly the elements of $\mathcal{S}_{\mathbf{R}}$ are satisfying the conditions laid in proposition \ref{rd21} and therefore it forms a commutative subring of $\mathcal{H}_{\mathbf{R}}$ and its group of units given by $U(\mathcal{S}_{\mathbf{R}})=\{ w+a(i+j+k) | w^{2}+3a^{2} \in U(\mathbf{R})\}$.
\begin{proposition}
$\mathcal{S}_{\mathbf{R}}$ is a commutative unital subring of $\mathcal{H}_{\mathbf{R}}$ which is not an ideal of $\mathcal{H}_{\mathbf{R}}$.
\end{proposition}
\begin{proof}
It is very straight forward to prove $\mathcal{S}_{\mathbf{R}}$ is a commutative unital subring of $\mathcal{H}_{\mathbf{R}}$. For each $r \in \mathcal{H}_{\mathbf{R}}$ and $q \in \mathcal{S}_{\mathbf{R}}$, $rq$ need not be in the form of elements in $\mathcal{S}_{\mathbf{R}}$. Hence the proof. 
\end{proof}
\begin{theorem} \label{rd22}
$\mathcal{S}_{\mathbf{Z}_{n}}$ is a field if and only if $n>3$, $n$ is a prime with $n-3$ has no square roots in $\mathbf{Z}_{n}$. Also, \\
 $ \begin{array}{ll}
    |U(\mathcal{S}_{\mathbf{Z}_{n}})| &= n^{2}-1 \quad \ \ \hbox{if $n$ is a prime and $(n-3)$ does not have  a square root in $\mathbf{Z}_{n}$;}  \\
          & =(n-1)^{2} \quad \hbox{if $n$ is a prime and $(n-3$) has  square roots in $\mathbf{Z}_{n}$.}  
   \end{array}$
\end{theorem}
\begin{proof}
Suppose  $n>3$, $n$ is a prime with $n-3$ has no square roots in $\mathbf{Z}_{n}$. Since $n$ is prime, $0$ is the only zero divisor in $\mathbf{Z}_{n}$. For zero divisors in $\mathcal{S}_{\mathbf{Z}_{n}}$, $w^{2}+3a^{2}=0$. This $\implies$ $w^{2}=-3a^{2}=(n-3)a^{2}$ in $\mathbf{Z}_{n}$. Since $n-3$ has no square roots in $\mathbf{Z}_{n}$, the only solution of the equation $w^{2}=(n-3)a^{2}$ is $w=0$ and $a=0$. That is, the only zero divisor in $\mathcal{S}_{\mathbf{Z}_{n}}$ is $0=0+0(i+j+k)$. Therefore $\mathcal{S}_{\mathbf{Z}_{n}}$ is a field and 
 $|U(\mathcal{S}_{\mathbf{Z}_{n}})|= n^{2}-1$ as $|\mathcal{S}_{\mathbf{Z}_{n}}|=n^{2}$.\\
Suppose  $n>3$, $n$ is a prime with $n-3$ has  a square root in $\mathbf{Z}_{n}$, say, $g$. Since $n$ is prime, $0$ is the only zero divisor in $\mathbf{Z}_{n}$. This $\implies$ $w^{2}=-3a^{2}=(n-3)a^{2}=g^{2}a^{2}$ in $\mathbf{Z}_{n}$. Thus we have $w=ga, (n-g)a$, for $a=0,1,2,3,...,(n-1)$. When $a=0$, $w$ has only one value and for other $(n-1)$ values of $a$, $w$ will have two values. \\
 $\therefore \quad $ $|U(\mathcal{S}_{\mathbf{Z}_{n}})|= n^{2}-(2(n-1)+1)=(n-1)^{2}$. 
\end{proof}
\begin{conjecture} \label{rd23}
Let $p$ be a prime and $e>0$ be an integer. Then $|U(\mathcal{S}_{\mathbf{Z}_{p^{e}}})|=p^{2(e-1)} \times |U(\mathcal{S}_{\mathbf{Z}_{p}})|$.
\end{conjecture}
Assuming that the above conjecture is true. Then we will have the following theorem. 
\begin{theorem} \label{rd24}
If $n=p_{1}^{e_{1}}p_{2}^{e_{2}} \cdot p_{l}^{e_{l}}$, then $|\mathcal{S}_{\mathbf{Z}_{n}}|=\prod_{r=1}^{l} p_{r}^{2(e_{r}-1)} \times |U(\mathcal{S}_{\mathbf{Z}_{p_{r}}})|$.
\end{theorem}
\begin{proof}
Using the group version of Chinese Remainder Theorem, we have $$\mathbf{Z}_{n} \cong \mathbf{Z}_{p_{1}^{e_{1}}} \times \mathbf{Z}_{p_{2}^{e_{2}}} \times \cdots \times \mathbf{Z}_{p_{l}^{e_{l}}}$$
and 
$$ |U(\mathbf{Z}_{n})|= |U(\mathbf{Z}_{p_{1}^{e_{1}}})| \times |U(\mathbf{Z}_{p_{2}^{e_{2}}})| \times \cdots \times |U(\mathbf{Z}_{p_{l}^{e_{l}}})|. $$
The above isomorphism will induce another isomorphism $$\mathcal{S}_{\mathbf{Z}_{n}} \cong \mathcal{S}_{\mathbf{Z}_{p_{1}^{e_{1}}}} \times \mathcal{S}_{\mathbf{Z}_{p_{2}^{e_{2}}}} \times \cdots \times \mathcal{S}_{\mathbf{Z}_{p_{l}^{e_{l}}}}$$ and it is defined by  \\ $f(w+a(i+j+k))=(w \ mod p_{1}^{e_{1}}+a \ mod p_{1}^{e_{1}}(i+j+k), w \ mod p_{2}^{e_{2}}+a \ mod p_{2}^{e_{2}}(i+j+k), \cdots, w  \ mod p_{l}^{e_{l}}+a \ mod p_{l}^{e_{l}}(i+j+k))$ \\
and this will in turn induce another isomorphism $$U(\mathcal{S}_{\mathbf{Z}_{n}}) \cong U(\mathcal{S}_{\mathbf{Z}_{p_{1}^{e_{1}}}}) \times U(\mathcal{S}_{\mathbf{Z}_{p_{2}^{e_{2}}}}) \times \cdots \times U(\mathcal{S}_{\mathbf{Z}_{p_{l}^{e_{l}}}})$$
$$ \therefore \quad |U(\mathcal{S}_{\mathbf{Z}_{n}})| \cong |U(\mathcal{S}_{\mathbf{Z}_{p_{1}^{e_{1}}}})| \times |U(\mathcal{S}_{\mathbf{Z}_{p_{2}^{e_{2}}}})| \times \cdots \times |U(\mathcal{S}_{\mathbf{Z}_{p_{l}^{e_{l}}}})|. $$
The theorem will follow immediately by applying the conjecture \ref{rd23}.
\end{proof}
\begin{corollary}
If $n=2^{s}3^{t}$ with integers $s,t \geq 0$, $|U(\mathcal{S}_{\mathbf{Z}_{n}})| =n\phi(n).$
\end{corollary}
\begin{proof}
Using the theorem \ref{rd24}, we have $|U(\mathcal{S}_{\mathbf{Z}_{n}})|=2^{2(s-1)} \times 2 \times 3^{2(t-1)} \times 6= 2^{2s} \times 3^{2s-1}=(2^{s}\times 3^{t}) \times (2^{s} \times 3^{t-1})=n \phi(n).$
\end{proof}
\begin{corollary}
If $n=p_{1}^{e_{1}}p_{2}^{e_{2}} \cdots p_{l}^{e_{l}}$ with $3<p_{1}<p_{2}< \cdots <p_{l}$ and if all $p_{r}-3$ have square roots in $\mathbf{Z}_{p_{r}}$ for $r=1,2, \cdots l, $ then $|U(\mathcal{S}_{\mathbf{Z}_{n}})|=(\phi(n))^{2}.$
\end{corollary}
\begin{proof}
By theorem \ref{rd24}, $|\mathcal{S}_{\mathbf{Z}_{n}}|=\prod_{r=1}^{l} p_{r}^{2(e_{r}-1)} \times |U(\mathcal{S}_{\mathbf{Z}_{p_{r}}})|$. Since every $p_{r}-3$ has square roots in $\mathbf{Z}_{p_{r}}$, $|U(\mathcal{S}_{\mathbf{Z}_{p_{r}}})|=(p_{r}-1)^{2}$ by theorem \ref{rd22}. Rearranging the factors, we get 
$|U(\mathcal{S}_{\mathbf{Z}_{n}})|=\dfrac{n^{2}}{(p_{1}p_{2} \cdots p_{l})^{2}} \times (p_{1}-1)^{2} \times (p_{2}-1)^{2} \cdots (p_{l}-1)^{2}=(\phi(n))^{2}$
\end{proof}
\begin{example}
This example deals with the cardinality of $U(\mathcal{S}_{\mathbf{Z}_{n}})$ for selective integral values of $n$. For details, see programme $3$ in appendix. \\
When $n=5$. \\
$0^{2}=0, 1^{2}=1,2^{2}=4,3^{2}=4,4^{2}=1$ indicate that $5$ is a prime with $5-3=2$  does not have a square root in $\mathbf{Z}_{5}$ and the only zero divisor is zero in $\mathbf{Z}_{5}$.
So, the number of zero divisors in $\mathcal{S}_{\mathbf{Z}_{n}}$ is the no. of solutions of the equation $w^{2}+3a^{2 }=0$ which implies $w^{2}=-3a^{2}=(5-3)a^{2}=2a^{2}$. This equation has only one solution and therefore $\mathcal{S}_{\mathbf{Z}_{n}}$ is a field and $|U(\mathcal{S}_{\mathbf{Z}_{n}})|=5^{2}-1=24$. \\
When $n=7$. \\
$0^{2}=0, 1^{2}=1,2^{2}=4,3^{2}=2,4^{2}=2, 5^{2}=4, 6^{2}=1$ indicate that $7$ is a prime with $7-3=4$  has  square roots 2 and 5  in $\mathbf{Z}_{7}$ and the only zero divisor is zero in $\mathbf{Z}_{7}$.
So, the number of zero divisors in $\mathcal{S}_{\mathbf{Z}_{n}}$ is the no. of solutions of the equation $w^{2}+3a^{2 }=0$ which implies $w^{2}=-3a^{2}=(7-3)a^{2}=4a^{2}$ and $w=2a,\ 5a$. This equation has $(1+2\times (7-1))=13$ and therefore $\mathcal{S}_{\mathbf{Z}_{n}}$ is not a field and $|U(\mathcal{S}_{\mathbf{Z}_{n}})|=7^{2}-13=36=(7-1)^{2}$. \\
When $n=6=2^{1}\times 3^{1}.$ \\
$0^{2}=0, 1^{2}=1,2^{2}=4,3^{2}=3,4^{2}=4, 5^{2}=1$ indicate that$\sqrt{0}=0, \sqrt{1}=1,5, \sqrt{4}=2,4, \sqrt{3}=3$ in $\mathbf{Z}_{6}$.The zero divisors in $\mathbf{Z}_{6}$ are $0,2,3,4$ and the units are $1,5$ in $\mathbf{Z}_{6}$. 
So, the number of units in $\mathcal{S}_{\mathbf{Z}_{n}}$ is the no. of solutions of the equation $w^{2}+3a^{2 }=1$ and $w^{2}+3a^{2}=5$ which imply $w^{2}=-3a^{2}+1=(6-3)a^{2}+1=3a^{2}+1$ and $w^{2}=3a^{2}+5$. These equations have $12$ solutions and therefore  $|U(\mathcal{S}_{\mathbf{Z}_{n}})|=12=6 \times 2=n \phi(n)$. \\
When $n=18=2^{1}\times 3^{2}.$ \\
$0^{2}=0, 1^{2}=1,2^{2}=4,3^{2}=9,4^{2}=16, 5^{2}=7, 6^{2}=0, 7^{2}=13, 8^{2}=10, 9^{2}=9, 10^{2}=10, 11^{2}=13, 12^{2}=0, 13^{2}=7, 14^{2}=16, 15^{2}=9, 16^{2}=4, 17^{2}=1$ indicate that $\sqrt{0}=0,6,12, \sqrt{1}=1,17, \sqrt{4}=2,16, \sqrt{9}=3,9,15, \sqrt{16}=4,14, \sqrt{7}=5,18,  \sqrt{13}=7,1, \sqrt{10}=8,10, \sqrt{}$ in $\mathbf{Z}_{6}$.The units in $\mathbf{Z}_{18}$ are $1,5,7,11,13,17$. 
So, the number of units in $\mathcal{S}_{\mathbf{Z}_{n}}$ is the no. of solutions of the equation $w^{2}+3a^{2 }=1$, $w^{2}+3a^{2}=5$, $w^{2}+3a^{2 }=7$, $w^{2}+3a^{2}=11$ , $w^{2}+3a^{2 }=13$, $w^{2}+3a^{2}=17$.  These equations have $108$ solutions and therefore  $|U(\mathcal{S}_{\mathbf{Z}_{n}})|=108=18 \times 6=n \phi(n)$. \\
\end{example}
\section{Halidon Rings}
In this section, we state some new results and the results essential for the construction of ring of quaternions which are also halidon rings. This is a rich class of noncommutative halidon rings. 
 The readers who are interested in the properties of halidon ring, can refer to \cite{at}, \cite{ath} and \cite{ath1}. \\
A primitive $m^{th}$ root of unity in a ring with unit element is completely different from that of in a field, because of the presence of nonzero zero divisors.  So we need a separate definition for a primitive $m^{th}$ root of unity. An element $\omega $ in a ring $R$ is called a  \textit{primitive} $m^{th}$ root if $m$ is the least positive integer such that  $\omega^{m}=1$ and
\begin{eqnarray*}
\sum_{r=0}^{m-1} \omega^{r(i-j)}&=& m, \quad  i= j (\ mod \ m )\\ &=& 0, \quad  i\neq j (\ mod \ m ). \end{eqnarray*} \\
More explicitly,
\begin{eqnarray*}
1+ \omega^{r}+(\omega^{r})^{2}+(\omega^{r})^{3}+(\omega^{r})^{4}+......+(\omega^{r})^{m-1}&=& m, \quad  r=0 \\ &=& 0, \quad 0<r\leq m-1. \end{eqnarray*} \\
A ring $R$ with unity is called a \textit{halidon} ring with index $m$ if there is a  primitive $m^{th}$ root of unity and $m$ is invertible in $R$. 
The following theorem will give the necessary and sufficient conditions for a ring to be a halidon ring. 
\begin{theorem} \label{rd2} (A. Telveenus \cite{ath})
A finite commutative ring $R$ with unity is a halidon ring with index $m$ if and only if there is a primitive $m^{th}$ root of unity  $\omega$ such that $m$, $\omega^{d}-1 \in U(R)$; the unit group of $R$ for all divisors $d$ of $m$ and $d<m$.
\end{theorem}
\begin{proposition} \label{rd16}
The homomorphic image of a commutative halidon ring with index $m$ and primitive $m^{th}$ of unity $\omega$ is also a halidon ring with index $m$.
\end{proposition}
\begin{proposition}\label{rd20}
Let R be a commutative halidon ring with index m and let k $>$ 1 be a divisor of m. Then R is also a halidon ring with index k.
\end{proposition}
In the rest of the section, let $R=\mathbb{Z}_{n}$ be a halidon ring with index $m$ and primitive $m^{th}$ root of unity $\omega$.
\begin{lemma} \label{rd3}
Let $p$ be an odd prime number and $k$ a positive integer. Then
\begin{enumerate}
  \item $U(\mathbb{Z}_{p})=<\omega>$ for some $\omega \in U(\mathbb{Z}_{p})$ with order $p-1$,
  \item $U(\mathbb{Z}_{p^{k}})=<\omega>$ for the same $\omega$ treating as an element in  $U(\mathbb{Z}_{p^{k}})$ with order $\phi(p^{k})$.
\end{enumerate}
\end{lemma}
\begin{definition}
Let R be a ring and $\alpha \in R$. The element $\alpha$ is called a \textit{primitive element} or \textit{primitive root} if $\alpha$ multiplicatively generates the unit group $U(R)$
 of the ring R.
\end{definition}
\begin{example}
$U(\mathbb{Z}_{5})=<2>$ with order 4  and $U(\mathbb{Z}_{5^{3}})=<2>$ with order $\phi(125)=100$. Clearly, $2 \in \mathbb{Z}_{5^{3}}$ is a primitive root but not a primitive root of unity in $\mathbb{Z}_{5^{3}}$.
\end{example}
\begin{proposition}  \label{rd17}
Let $p$ be an odd prime number. Then $\mathbb{Z}_{p^{k}}$ is a halidon ring with index $m=p-1$ and  $\omega_{1}=\omega^{p^{k-1}}$ is a primitive $m^{th}$ root of unity for positive integers $k\geq 1$.
\end{proposition}
The complete characterisation of the halidon property in $\mathbb{Z}_{n}$, where $n$ is odd, is given by the following theorem.
\begin{theorem}\label{rd18}
The ring $\mathbb{Z}_{n}$, where $ \displaystyle n=p_{1}^{e_{1}}p_{2}^{e_{2}}p_{3}^{e_{3}}.....p_{k}^{e_{k}}$ with $2<p_{1}<p_{2}<.....<p_{k}$ is a halidon ring with index $m$ and the primitive $m^{th}$ root of unity $\omega$ if and only if each $ \displaystyle \mathbb{Z}_{p_{i}^{e_{i}}}$ is a halidon ring with index $m$ and primitive $m^{th}$ root of unity $ \displaystyle \omega_{i}=\omega \ mod \ {p_{i}^{e_{i}}} $ for each $i=1,2,3,...,k$.
\end{theorem}
\begin{definition} \label{rd7}
Let $p_{1},p_{2},p_{3},....,p_{k}$ be odd primes and let $\phi(x)$ be the Euler's totient function. We define the \textit{halidon function}  $$\psi(n)= \begin{cases} gcd \{ \phi(p_{1}^{e_{1}}), \phi(p_{2}^{e_{2}}),\phi(p_{3}^{e_{3}}),...., \phi(p_{k}^{e_{k}})\}, & n=p_{1}^{e_{1}}p_{2}^{e_{2}}p_{3}^{e_{3}}.....p_{k}^{e_{k}} \\
1, & n \ \text{is even}\end{cases} $$
\end{definition}
\begin{proposition}
Let $n$ be as in definition \ref{rd7}. Then the halidon function $$\psi(n)=gcd\{ p_{1}-1,p_{2}-1,p_{3}-1,....,p_{k}-1\},$$ which is independent of the exponents $e_{1},e_{2},e_{3},....,e_{k}$.
\end{proposition}
Now we state the following theorem, which was previously a conjecture(see \cite{ath2}):
 \begin{theorem}\label{rd1}(A. Telveenus \cite{ath2})
If $R=\mathbb{Z}_{n}$ and $n=p_{1}^{e_{1}}p_{2}^{e_{2}}p_{3}^{e_{3}}.....p_{k}^{e_{k}}$ with primes $p_{1}<p_{2}<p_{3}<....<p_{k}$ including 2, then $R$ is a halidon ring with maximum index $m_{max}=\psi(n)$.
\end{theorem}
From proposition \ref{rd17}, we have $\mathbb{Z}_{p^{k}}$ is a halidon ring with index $m=p-1$ and the primitive $m^{th}$ root of unity $\omega_{1}=\omega^{p^{k-1}}$ for positive integer $k\geq 1$ where $\omega$ is a primitive $m^{th}$ root of unity in $\mathbb{Z}_{p}$.
\begin{lemma} \label{rd11}
Let $p$ be a prime number. Then the number of primitive $k^{th}$ roots of unity in $\mathbb{Z}_{p}$ is $\phi(k)$.
\end{lemma}
\begin{theorem} \label{rd19}
Let $n=p_{1}^{e_{1}}p_{2}^{e_{2}}p_{3}^{e_{3}}.....p_{k}^{e_{k}}$ such that $e_{i}>0$ are integers and $p_{i}=mt_{i}+1;$ where $m$ is the maximum index of the halidon ring $\mathbb{Z}_{n}$ and $t_{i}$'s are relatively prime for all $i=1,2,3...,k$. Then the number of primitive $m^{th}$ root of unity in $\mathbb{Z}_{n}$ is $[\phi(m)]^{k}$.
\end{theorem}
The next example shows the existence of a class non commutative halidon rings using quaternions. 
\begin{example}
The ring $\mathcal{H}_{\mathbb{Z}_{7}}$ is a halidon ring with maximum index $48$ with a primitive $48^{th}$ root of unity $q = 6+ 6i+6j+4k$ and $48=6$ is invertible in $\mathcal{H}_{\mathbb{Z}_{7}}$.
\begin{verbatim}
q^1 = 6+ 6i+6j+4k,  q^1-1 is invertible in H
q^2 = 4+ 2i+2j+6k,  q^2-1 is invertible in H
q^3 = 4+ 1i+1j+3k,  q^3-1 is invertible in H
q^4 = 0+ 2i+2j+6k,  q^4-1 is invertible in H
q^5 = 1+ 5i+5j+1k,  q^5-1 is invertible in H
q^6 = 5+ 1i+1j+3k,  q^6-1 is invertible in H
q^7 = 6+ 1i+1j+3k,  q^7-1 is invertible in H
q^8 = 5+ 0i+0j+0k,  q^8-1 is invertible in H
q^9 = 2+ 2i+2j+6k,  q^9-1 is invertible in H
q^10 = 6+ 3i+3j+2k, q^10-1 is invertible in H
q^11 = 6+ 5i+5j+1k, q^11-1 is invertible in H
q^12 = 0+ 3i+3j+2k, q^12-1 is invertible in H
q^13 = 5+ 4i+4j+5k, q^13-1 is invertible in H
q^14 = 4+ 5i+5j+1k, q^14-1 is invertible in H
q^15 = 2+ 5i+5j+1k, q^15-1 is invertible in H
q^16 = 4+ 0i+0j+0k, q^16-1 is invertible in H
q^17 = 3+ 3i+3j+2k, q^17-1 is invertible in H
q^18 = 2+ 1i+1j+3k, q^18-1 is invertible in H
q^19 = 2+ 4i+4j+5k, q^19-1 is invertible in H
q^20 = 0+ 1i+1j+3k, q^20-1 is invertible in H
q^21 = 4+ 6i+6j+4k, q^21-1 is invertible in H
q^22 = 6+ 4i+4j+5k, q^22-1 is invertible in H
q^23 = 3+ 4i+4j+5k, q^23-1 is invertible in H
q^24 = 6+ 0i+0j+0k, q^24-1 is invertible in H
q^25 = 1+ 1i+1j+3k, q^25-1 is invertible in H
q^26 = 3+ 5i+5j+1k, q^26-1 is invertible in H
q^27 = 3+ 6i+6j+4k, q^27-1 is invertible in H
q^28 = 0+ 5i+5j+1k, q^28-1 is invertible in H
q^29 = 6+ 2i+2j+6k, q^29-1 is invertible in H
q^30 = 2+ 6i+6j+4k, q^30-1 is invertible in H
q^31 = 1+ 6i+6j+4k, q^31-1 is invertible in H
q^32 = 2+ 0i+0j+0k, q^32-1 is invertible in H
q^33 = 5+ 5i+5j+1k, q^33-1 is invertible in H
q^34 = 1+ 4i+4j+5k, q^34-1 is invertible in H
q^35 = 1+ 2i+2j+6k, q^35-1 is invertible in H
q^36 = 0+ 4i+4j+5k, q^36-1 is invertible in H
q^37 = 2+ 3i+3j+2k, q^37-1 is invertible in H
q^38 = 3+ 2i+2j+6k, q^38-1 is invertible in H
q^39 = 5+ 2i+2j+6k, q^39-1 is invertible in H
q^40 = 3+ 0i+0j+0k, q^40-1 is invertible in H
q^41 = 4+ 4i+4j+5k, q^41-1 is invertible in H
q^42 = 5+ 6i+6j+4k, q^42-1 is invertible in H
q^43 = 5+ 3i+3j+2k, q^43-1 is invertible in H
q^44 = 0+ 6i+6j+4k, q^44-1 is invertible in H
q^45 = 3+ 1i+1j+3k, q^45-1 is invertible in H
q^46 = 1+ 3i+3j+2k, q^46-1 is invertible in H
q^47 = 4+ 3i+3j+2k, q^47-1 is invertible in H
q^48 = 1+ 0i+0j+0k, q is of order 48
q is a primitive 48 th root of unity
\end{verbatim}
$\mathcal{H}_{\mathbb{Z}_{n}}$ with $n$ even are not halidon rings. See programme $2$ in the appendix which has been developed using the theorem $\ref{rd2}.$
\end{example}
\section{Conclusions}
Halidon rings over quaternions have broadened the scope of mathematical research and created useful connections across several fields. \\
•	Cryptography, where algebraic structures can support the design and analysis of secure systems. \\
•	Group algebras, where quaternion-based ring structures offer new avenues for theoretical investigation. \\
•	Permutation polynomials over halidon rings, where these rings provide a broader setting for studying polynomial behaviour and applications. \\

\section{Appendix}
Programme $1$; Calculation of nil potent elements in $\mathcal{H}_{\mathbb{Z}_{n}}$. The calculation of $q^{s}$ is based on Theorem $\ref{t1}$. 
\begin{verbatim}
#include<iostream>
#include<cmath>
using namespace std;
int main() {
	cout << "To find the nil potent elements  q=a+bi+cj+dk." << endl;
	long long int a, b, c, d, n, i, j = 1, k, s, r1[25000], r2[25000],
 r3[25000], r4[25000], s1[25000], y = 0;
	cout << "Enter n=" << endl;
	cin >> n;
	for (a = 0; a < n; ++a) {
		for (b = 0; b < n; ++b) {
			for (c = 0; c < n; ++c) {
				for (d = 0; d < n; ++d) {
					r1[1] = a; r2[1] = b; r3[1] = c; r4[1] = d;
					for (s = 1; s <n; ++s) {
				//		cout << "q^" << s << " = " << r1[s] << " + " << r2[s] << "i" <<
                              " + " << r3[s] << "j" << " + " << r4[s] << "k" << endl;
						if (r1[s] == 0 && r2[s] == 0 && r3[s] == 0 && r4[s] == 0)
						{
							cout << "q=" << a << "+" << b << "i + " << c << "j +  " << d << "k " 
                           << "is of order " << s << endl << endl;; ++y; break;
						}
						r1[s + 1] = (a * r1[s] + (n - b) * r2[s] + (n - c) * r3[s] + (n - d) * r4[s]) % n;
						r2[s + 1] = (b * r1[s] + a * r2[s] + d * r3[s] + (n - c) * r4[s]) % n;
						r3[s + 1] = (c * r1[s] + (n - d) * r2[s] + a * r3[s] + b * r4[s]) % n;
						r4[s + 1] = (d * r1[s] + c * r2[s] + (n - b) * r3[s] + a * r4[s]) % n;
					}
				}
			}
		} 
		} cout << "number of nil potent elements=" << y; return 0;
	}
\end{verbatim}
Programme $2$  to check whether a ring of quaternions over $\mathbb{Z}_{n}$ a non trivial halidon ring or not. 
\begin{verbatim}
#include<iostream>
#include<cmath>
using namespace std;
int main() {
	cout << "To check for non trivial halidon rings." << endl;
	long long int w, x, y,z,c=0, n, i, j = 1, k,l, p, p1,p2, a, b,s,
		r1[20000], r2[20000], r3[20000], r4[20000], s1[20000], max,c1=0,c2=0;
	cout << "Enter n=" << endl;
	cin >> n;
	for (w = 0; w < n; ++w) {
		for (x = 0; x < n; ++x) {
			for (y = 0; y < n; ++y) {
				for (z = 0; z < n; ++z) {
					r1[1] = w; r2[1] = x; r3[1] = y; r4[1] = z;
					p = (w * w + x * x + y * y + z * z) % n;
					cout << "p= " << p << endl;
					p1 = (x * x+ y * y + z * z) % n;
					for (i = 1; i < n; ++i) {
						a = (p * i) % n;
						if (a == 1)
						{
								for (s = 1; s < pow(n, 4); ++s) {
								cout << "q^" << s << " = " << r1[s] << "+ " << r2[s] << "i" << "+" <<
                                     r3[s] << "j" << "+" << r4[s] << "k" << endl;
								p2 = ((r1[s]-1)*(r1[s]-1)+r2[s]*r2[s]+r3[s]*r3[s]+r4[s]*r4[s]) % n;
								for (l = 1; l < n; ++l) {
					b = (p2 * l) % n; if (b == 1) { cout << "q^" << s << "-1 is invertible in H" << endl; 
										c1++; }
									}
								if (r1[s] == 1 && r2[s] == 0 && r3[s] == 0 && r4[s] == 0)
								{
									cout << "q is of order " << s << endl; ++c; break;
								}
								r1[s + 1] = (w * r1[s] + (n - x) * r2[s] + (n - y )* r3[s] + (n - z) * r4[s]) % n;
								r2[s + 1] = (x * r1[s] + w * r2[s] + z * r3[s] + (n - y) * r4[s]) % n;
								r3[s + 1] = (y * r1[s] + (n - z) * r2[s] + w * r3[s] + x * r4[s]) % n;
								r4[s + 1] = (z * r1[s] + y * r2[s] + (n - x) * r3[s] + w * r4[s]) % n;
							} if (c1 == s - 1) { cout << "q is a primitive " << s << " th root of unity." << endl; 
                              ++c2; }s1[j] = s; j = j + 1; c1 = 0;
						}
					}
				}
			}
		}
	} max = s1[1]; for (k = 1; k < j; ++k) { if (s1[k] > max) { max = s1[k]; } }
	cout << "Maximum order= " << max << endl;
	cout << "number of elements with finite order= " << c << endl;
	cout << "No. of non-trivial primitive roots=" << c2-1;
	return 0;
}\end{verbatim}
Programme $3$ to find the unit group of the subring $\mathcal{S}_{\mathbf{Z}_{n}}$ of $\mathcal{H}_{\mathbf{Z}_{n}}$.
\begin{verbatim}
#include<iostream>
#include<cmath>
using namespace std;
int gcd(int c, int d) {
	if (c == 0)
		return d;
	return gcd(d % c, c);
}
int main() {
	cout << "To find the unit group of the subring S_Zn of H_Zn" << endl;
	long long int u=0,z=0,l=0, a=0,i,j=1,x,n, W=0,w=0,r[2000],s[300][300];
	cout << "Enter n=" << endl;
	cin >> n;
	r[1]=0;
	for (x = 2; x < n ; x++) { if (gcd(x,n) != 1) { j++; r[j] = x; } }
	cout << "zero divisors are " << endl;
	for (i = 1; i < j + 1; i++) { cout << " " << r[i]; }
	cout << endl;
	for (l = 1; l < j+1; l++) {
		W = w * w;
		for (a = 0; a < n; a++) {
			W = ((2*n - 3) * (a * a) + r[l])%n; s[l][a] = W;  
           cout << "s[" << l << "]" << "[" << a << "]= " << s[l][a] << endl; 
			for (w = 0; w < n; ++w) { u = (w * w)%n;  if (s[l][a] == u) 
            { cout << "w= " << w << endl; z++; continue; } }
		}
	}
	cout << endl << "No. of zero divisors= " << z << endl; if (z == 1) 
    { cout << "The subring is a field."; }
	cout << endl << "No. of units= " << n*n-z;
	return 0;
}\end{verbatim}
\end{document}